\documentclass[a4paper, 11pt]{article}

\usepackage{graphicx}
\usepackage{bbm}
\usepackage{mathrsfs}
\usepackage{amsmath}
\usepackage{amssymb}
\usepackage{amsthm}
\usepackage[center]{caption}
\usepackage{enumerate}
\usepackage{verbatim}
\usepackage{authblk}
\usepackage{calc}
\usepackage{color}
\usepackage{longtable}
\usepackage{multirow}
\usepackage[font={small}]{caption}
\usepackage{centernot}

\theoremstyle{plain}
\newtheorem{thm}{Theorem}[section]
\newtheorem{cor}[thm]{Corollary}
\newtheorem{lem}[thm]{Lemma}
\newtheorem{prop}[thm]{Proposition}

\theoremstyle{definition}

\newtheorem*{remark}{Remark}

\renewcommand{\P}{\mathbb{P}}

\newcommand{\E}{\mathbb{E}}

\newcommand{\R}{\mathbb{R}}

\newcommand{\Px}{P^{\xi}}
\newcommand{\Ex}{E^{\xi}}

\newcommand{\hs}{\hspace{2mm}}
\newcommand{\hsl}{\hspace{1mm}}
\newcommand{\1}{1\hspace{-0.098cm}\mathrm{l}}
\newcommand{\ind}{\1} 

\newcommand{\eps}{\varepsilon}
\newcommand{\Z}{\mathbb{Z}}

\newcommand{\Prob}{\mathrm{Prob}}

\newcommand{\Om}{\Omega}

\newcommand{\p}{\P}

\newcommand{\be}[1]{\begin{equation}\label{#1}}
\newcommand{\ee}{\end{equation}}
\newcommand{\ra}{\rightarrow}
\newcommand{\bal}{\begin{aligned}}
\newcommand{\eal}{\end{aligned}}

\newcommand{\F}{\mathcal{F}}
\newcommand{\calG}{\mathcal{G}}

\newcommand{\smfr}[2]{\hbox{$\frac{#1}{#2}$}}

\author{Marcel Ortgiese and Matt Roberts}

\begin{document}

\begin{center}
{\Large \bf One-point localization for  branching random walk\\[1mm] in Pareto environment}\\[5mm]
\vspace{0.4cm}
\textsc{
Marcel Ortgiese\footnote{Institut f\"ur Mathematische Statistik,
Westf\"alische Wilhelms-Universit\"at M\"unster,
Einsteinstra\ss{}e 62,
48149 M\"unster,
Germany. } and 
Matthew I. Roberts\footnote{Department of Mathematical Sciences,
University of Bath,
Claverton Down,
Bath BA2 7AY,
United Kingdom.}
} 
\\[0.8cm]
{\small \today} 
\end{center}

\vspace{0.3cm}

\begin{abstract}\noindent 
We consider a branching random walk on the lattice, where the branching rates are given by an i.i.d.\ Pareto random potential. 
We show a very strong form of intermittency, where with high probability most of the mass
of the system 
is concentrated in a single site with high potential.
The analogous one-point localization is already known for the parabolic Anderson model, which describes the 
expected number of particles in the same system. In our case,  we rely on very fine estimates for the behaviour of particles 
near a good point.
This complements our earlier results that in the rescaled picture most of the mass is concentrated on a small island.
  \par\medskip

  
  \end{abstract}

\section{Introduction and main result}

\subsection{Introduction}

We consider a branching process in random environment defined on $\Z^d$. 
We start the system with a single particle at the origin, which 
can branch and also migrate in space.
Given the random potential $\xi = \{ \xi(z) \, : \, z  \in \Z^d\}$ of non-negative random variables,
 a particle splits into two particles at rate $\xi(z)$ when at site $z$. 
Furthermore, each particle moves independently according to the law of a nearest neighbour simple random walk in continuous time on $\Z^d$. This particular model was introduced in~\cite{GM90}.

Most of the analysis of this model has concentrated on the
expected number of particles. 
We fix a realization of the environment $\xi$ and denote the expected number of particles by
\[ u(z,t) = E^\xi [ \# \{ \mbox{particles at site } z \mbox{ at time } t  \} ] , \]
where the expectation $E^\xi$ is only over the branching and migration mechanisms
and $\xi$ is kept fixed.
Then $u(z,t)$ solves the 
stochastic partial differential equation, known as the \emph{parabolic Anderson model} (PAM),
\[ \begin{aligned} \partial_t u(z,t) &  = \Delta u(z,t) + \xi(z) u(z,t) , & \quad \mbox{for  }z \in \Z^d, t \geq 0, \\
u(z,0) & = \1_{\{z = 0 \}} &\quad \mbox{for } z \in \Z^d . 
\end{aligned}\]
Here, $\Delta$ is the discrete Laplacian defined for any function $f : \Z^d \ra \R$ as
\[ \Delta f(z) = \sum_{y \sim z} (f(y) - f(z)) , \quad z \in \Z^d, \]
where we write $y \sim z$ if $y$ is a neighbour of $z$ on the lattice $\Z^d$. 

The analysis of this model has a long history, which in the mathematics literature began with the work of~\cite{GM90}.
The central observation that has driven much of the research on the PAM
is that the system exhibits \emph{intermittency}: most of the mass is concentrated
on a small number of islands with high potential.
This effect is well understood for the PAM: see the surveys~\cite{GK05, M11, K16}.
In particular, the number and size of the relevant islands depends
on the decay of $\p ( \xi(0) > x)$ as $x \ra \infty$.
For bounded potentials the islands grow with $t$, and there
is an intermediate regime for double exponentially distributed potential 
with islands of finite size. Finally, for any potential with heavier tails, a single
island of a single point carries most of the mass.

We are particularly interested in the most extreme case, when 
the potential is Pareto distributed. 
In this case, the evolution of the PAM is particularly well understood, 
including asymptotics for the total masses, one point localization and a
scaling limit: 
see~\cite{HMS08, KLMS09, MOS11,OR14}.

In general much less known is about the branching system itself (without taking expectations). 
Some of the earlier results include~\cite{ABMY00} and~\cite{GKS13},
who look at the asymptotics of the expectation (with respect to $\xi$) of higher moments of the number of 
particles.

The starting point for this article is our recent 
result~\cite{OR14}.
We showed that---in the Pareto case---the hitting times of sites, the number of particles, and 
the support in a appropriately rescaled system are well described by a
process defined purely in terms of the 
environment $\xi$ (that is, given $\xi$, the process is deterministic),  which we call the lilypad process.

As one of the  applications of these results, we deduced that, 
under the right rescaling, the process is concentrated on a small island. 
However, without rescaling, the radius of the relevant
island at time $T$ is still roughly of order $T^{\frac{\alpha}{\alpha-d}}$, i.e.\ growing on a  superlinear scale.
In this paper, we address the question of whether these bounds can be improved
to a comparable scale to the results for the PAM. 
Indeed, our main result shows that the number of particles in the 
branching system is also localized in a single point.

\subsection{Main result}

We assume from now on that $\{ \xi(z),\, z \in\Z^d\}$
is a collection of independent and identically distributed Pareto random variables. 
To be precise, if we denote the underlying probability measure by $\Prob$, then we take
\[ \Prob ( \xi(z) > x ) = x^{- \alpha}  \quad \mbox{for all } x \geq 1 , \]
for a parameter $\alpha > 0$ and any $z \in \Z^d$. 
We will also assume that $\alpha > d$, which is known to be necessary 
for the total mass of
the PAM to remain finite (see \cite{GM90}).

For a fixed environment $\xi$, we denote by $P_y^\xi$  the law of the branching simple random walk in continuous time with binary branching and branching rates $\{\xi(z)\, , \, z \in \Z^d\}$ started with a single particle at site $y$.
Finally, for any measurable set $F \subset \Om$, we define
\[ \p_y ( F \times \cdot) = \int_F P_y^\xi ( \cdot) \,\Prob(d \xi) . \]
If we start with a single particle at the origin, we omit the subscript $y$ and simply write
$P^\xi$ and $\P$ instead of $P_0^\xi$ and $\p_0$.

We define
$Y(z,t)$ to be the set of particles at the point $z$ at time $t$, and let $Y(t)$  be the set of all particles present at time $t$.
We are interested in the number of particles
\[N(z,t) = \# Y(z,t)\quad\mbox{and}\quad N(t) = \# Y(t) = \sum_{z\in\Z^d} N(z,t). \]

Our main result states that the system is intermittent with 
one relevant island consisting of a single point.

\begin{thm}\label{oneptthm}
There exists a stochastic process $(w_t)_{t \ge 0}$ such
that 
\[ \frac{N(w_t,t)}{N(t)} \to  1\]
in probability under $\P$ as $t \ra \infty$.
\end{thm}

This theorem says that at any large time, with probability close to 1, the overwhelming majority of the particles are situated at exactly one site. Recall that in \cite{OR14}, we showed that almost all the particles are contained in a ``small'' ball around one site. The bounds in \cite{OR14} told us approximately what time each site was hit, and approximately how many particles were at each site, but were not precise enough to prove Theorem~\ref{oneptthm}.

In fact, in proving Theorem \ref{oneptthm}, we will use the results in \cite{OR14} to reduce the number of sites that we need to look at; and since we used a rescaled picture in that article, we will again use a rescaling here, to prove a slightly different statement from Theorem \ref{oneptthm} that implies the statement above.

\begin{remark} \emph{Comparison with the PAM}. 
For the PAM, \cite{KLMS09} show that the analogous statement of Theorem~\ref{oneptthm}
holds for the expected number of particles. Moreover, 
they show that for any time $t$, almost surely the expected number of particles 
is \emph{almost surely} concentrated in at most two sites. 
For the branching system (without taking expectations), a statement about the almost sure behaviour is beyond the scope
of our current techniques. We hope to address this in future work. 
However, we can compare the maximizing site in our statement with the 
maximizing site in the PAM. Indeed, we will show that the maximizer is the maximizing site
of the lilypad model defined in~\cite{OR14}. Moreover, these two sites (the maximizing site in the lilypad model, and the maximizing site in the PAM) were already compared in~\cite[Thm.\ 1.5(ii)]{OR14}, where we showed that at any time, the two sites agree with probability bounded away from $0$ and $1$.

It is worth noting too that our methods could be used to give a relatively short proof that the solution of the PAM is localized in one point with high probability.
\end{remark}

In a companion paper~\cite{brwre_scaling}, we show that the rescaled  branching system converges to a Poissonian system, and as a corollary,  we deduce that the maximizing site has the \emph{ageing} property.

\subsection{The Many-to-One or Feynman-Kac formula}
We suppose that under $P^\xi_y$, we have a simple random walk $(X(t))_{t\geq0}$, started from $y$, independent of the environment and of the branching random walk above.

Given a particle $v \in Y(t)$, and $s \leq t$, we write $X_v(s)$ for the position of the unique ancestor of $v$ that was present at time $s$.

We recall the Feynman-Kac formula or, as we often refer to it, the many-to-one formula. This simple result is key to our analysis, and we will use it regularly.

\begin{lem}[Many-to-one / Feynman-Kac formula]\label{le:mto}
If $f$ is measurable, then $\Prob$-almost surely, for any $s>0$,
\[\Ex\bigg[\sum_{v\in Y(s)} f((X_v(u))_{u\in[0,s]})\bigg] = \Ex\bigg[\exp\left(\int_0^s \xi(X(u)) du\right) f((X(u))_{u\in[0,s]})\bigg].\]
\end{lem}

\subsection{Outline of the proof and layout of the paper}\label{sec:heur}

We now explain why Theorem \ref{oneptthm} holds. We already know that almost all the particles are in a small ball around the ``good site'' $w$, so really we only need to consider the behaviour of particles within that ball. Imagine, for a moment, that we begin with one particle at $w$. The site $w$ has very large potential, on a much larger scale than the sites around it, so to simplify the picture we imagine particles breeding only when at $w$. Then the Feynman-Kac formula tells us that, under this simplification, the expected population size at $z$ at time $t$ is
\[\Ex_w[e^{\int_0^t \xi(w)\ind_{\{X(s)=w\}}ds}\ind_{\{X(t)=z\}}].\]

Taking $z=w$ we see that the expected population size at $w$ grows at least like $e^{\xi(w)t - 2dt}$, where the $e^{-2dt}$ factor corresponds to the probability that the random walk remains at $w$ up to time $t$. We state this precisely in Lemma \ref{tr}. In fact, standard methods can be used to show that in fact the population size at $w$ (not just the \emph{expected} population size) grows at least like $e^{\xi(w)t - 2dt}$, up to some small error.

Now let $z$ instead be a neighbour of $w$. To be at $z$ at time $t$, the random walk $(X(s))_{s\le t}$ must spend some time away from $w$. Spending time $s$ away from $w$ costs $e^{-\xi(w)s}$, which suggests we should make $s$ small; but jumping within a time interval of length $s$ costs at least $s$. A simple optimization calculation suggests that we should choose $s$ roughly of the order $1/\xi(w)$. We therefore guess that the size of the population at any neighbour of $w$ should be at most $1/\xi(w)$ times as large as the size of the population at $w$ itself. Similarly, at distance 2 from $w$ we expect to see about $1/\xi(w)^2$ times as many particles as at $w$, and so on. Since $\xi(w)\gg 1$, and there are at most $(2d)^k$ sites at distance $k$ from $w$, the population at $w$ is much larger than the total population elsewhere. This statement will be made precise in Lemma \ref{ybest}.

This is the basic idea behind the result, but unfortunately there are technical difficulties. The main problem we face is that we do not only have to consider starting with one particle at $w$. We have particles arriving at $w$ from other sites at unpredictable times: it may be that the first particle to hit $w$ was ahead of its time, and there is a wait before more particles pour in; or it may be that immediately after the first particle hits $w$, a huge number more arrive hot on its tail. Besides, with huge numbers of particles to consider, \emph{some} particles will show unusual behaviour. For example, some might visit $w$ only for a very brief period, and end up with more descendants at a neighbouring site than at $w$.

We take each particle to arrive at $w$ (whose ancestors have not already visited $w$) in turn. The argument above is enough to show that the expected number --- and therefore, by Markov's inequality, the actual number --- of descendants at sites other than $w$ is much smaller than the expected number at $w$. This is Proposition \ref{notyprop}. The challenge then is to show that the number of particles at $w$ is as large as it should be. We do this by first showing that each particle arriving at $w$ has, with fairly large probability, a reasonable number of descendants at $w$ at time $t$ (not too much smaller than the expected number). This is done in Lemmas \ref{pz} and \ref{onev}. Say that such particles ``behave well''.

We then break the time interval $[0,t]$ up into small chunks. Provided that a lot of particles arrive at $w$ within a chunk, the probability that at least half behave well is extremely close to 1. And if at least half behave well, then we have a good number of descendants at $w$ at time $t$. This is carried out in Lemma \ref{jbulk}. There are some further technicalities that are taken care of in Lemmas \ref{tsmall} and \ref{smalllem}, but this is essentially enough to show that the population at $w$ is as large as it should be (Proposition \ref{bigenough}).

Combining Propositions \ref{notyprop} and \ref{bigenough} does most of the work in proving Theorem \ref{oneptthm}. However, just as in \cite{OR14}, we have to check that various events of small probability do not occur. This is done in Section \ref{mainproofsec}.

\subsection{Some definitions}

In order to apply the heuristic above, we will need to ensure that particles cannot visit points of large potential other than $w$. This motivates the following definitions. We now work with general points $y$ and $z$, but it may be helpful to imagine $y$ as the good point $w$, and $z$ as a point nearby.

For $U\subset \Z^d$, $y,z\in\Z^d\setminus U$ and $t\geq 0$, define
\[Y(z,t;U) = \{v\in Y(z,t) : \not\exists s\leq t \hbox{ with } X_v(s)\in U\}\]
and
\[Y(z,t;U,y) = \{v\in Y(z,t) : \exists s\leq t \hbox{ with } X_v(s)=y, \hbox{ but} \not\exists s\leq t \hbox{ with } X_v(s)\in U\},\]
and let $N(z,t;U) = \# Y(z,t;U)$ and $N(z,t;U,y) = \# Y(z,t;U,y)$. For $v\in\bigcup_{s\leq t} Y(s)$, let $Y^v(z,t;U) = \{w\in Y(z,t;U) : v\leq w\}$, the set of descendants of $v$ that are in $Y(z,t;U)$, and $N^v(z,t;U) = \# Y^v(z,t;U)$; similarly for $Y^v(z,t;U,y)$ and $N^v(z,t;U,y)$. Also define the event
\[A(z,t;U) = \{X(t) = z, \hs \not\exists s\leq t \hbox{ with } X(s)\in U\}.\]
Then for $\theta\geq0$, define
\[U_{y,\theta} = \{z\in \Z^d : \xi(z)> \xi(y)-\theta\}\setminus\{y\}\]
and
\[f_\theta(y,t) = \Ex_y[N(y,t;U_{\theta,y})].\]

One complication is that there are always particles arriving at $w$ from elsewhere, and as soon as a new particle arrives it begins contributing to the population at $w$. Controlling this precisely is difficult, and if for example a large number of particles arrive at $w$ just after $w$ is first hit, then the population at later times will be much larger than if just one particle hits $w$ significantly before any others. (Even very small fluctuations can be significant when the potential at $w$ is so large.) Instead of tackling this issue, we instead work on estimating the population at and near $w$ given the information about particles arriving at $w$ for the first time. Again we make our definitions for a general site $y$.

Let $H(y) = \inf\{t>0 : N(y,t) \ge 1\}$, the first hitting time of $y$. For $v\in \bigcup_{t\geq0} Y(t)$ and $y\in\Z^d$, let $\tau_y(v) = \inf\{s\geq0 : X_v(s) = y\}$, the first time that the particle $v$ hits the point $y$. Define
\[L_\theta(y,t) = \bigg\{v\in \bigcup_{s\leq t} Y(y,s) \bigg|  \not\exists w<v  : \, w\in \bigcup_{s\leq t} Y(y,s)
\ \mbox{ and}\ \not\exists s \leq \tau_y(v) : \, X_v(s) \in U_{y,\theta}\bigg\}.\]
That is, $L_\theta(y,t)$ contains only those particles that hit $y$ before $t$, but whose ancestors have not hit $y$ or visited $U_{y,\theta}$. Let $\calG_{L_\theta(y,t)}$ be the $\sigma$-algebra that contains information about which particles are in $L_\theta(y,t)$ as well as the times that they hit $y$,
\[\calG_{L_\theta(y,t)} = \sigma(L_\theta(y,t), \{\tau_y(v) : v\in L_\theta(y,t)\}).\]

If $y$ has much larger potential than any nearby point, then the number of particles that we anticipate seeing at $y$ at time $t$ is (for any $\theta\ll \xi(y)$)
\[\sum_{v\in L_\theta(y,t)}f_\theta(y,t-\tau_y(v)),\]
and by the heuristic in Section \ref{sec:heur}, we anticipate seeing at most a constant times
\[\frac{1}{\xi(y)}\sum_{v\in L_\theta(y,t)}f_\theta(y,t-\tau_y(v))\]
particles elsewhere. The following two sections make these statements precise, leaving a small amount of room for inaccuracies. The first, Section \ref{notysec}, shows that the population away from $y$ is unlikely to be bigger than $\xi(y)^{-9/10}\sum_{v\in L_\theta(y,t)}f_\theta(y,t-\tau_y(v))$, and the second, Section \ref{ysec}, shows that the population at $y$ is unlikely to be smaller than $\xi(y)^{-4/5}\sum_{v\in L_\theta(y,t)}f_\theta(y,t-\tau_y(v))$. By the argument above, these are perfectly reasonable statements, and they of course imply that the population at $y$ is bigger than the population away from $y$.

Throughout this paper, $|\cdot |$ will denote the $L^1$-norm on $\R^d$ and $B(x,r)$ the open 
ball around $x$ of radius $r$ in $L^1$. 

\section{The population away from $y$ is not too big}\label{notysec}

Our main result in this section is the following.

\begin{prop}\label{notyprop}
Take $y\in\Z^d$ with $\xi(y)\geq 2$ and suppose that $\theta > 10d\xi(y)^{19/20}$. Then for any $t>0$,
\[\Px\bigg(\sum_{v\in L_\theta(y,t)}\sum_{z\neq y} N^v(z,t;U_{y,\theta}) \geq \sum_{v\in L_\theta(y,t)}\frac{f_\theta(y,t-\tau_y(v))}{\xi(y)^{9/10}} \hsl\bigg|\hsl\calG_{L_\theta(y,t)}\bigg) \leq \frac{2d\xi(y)^{-1/20}}{(1-2^{-19/20})^2}\]
almost surely.
\end{prop}

The precise value on the right-hand side is not important; we only need that it is small when $\xi(y)$ is large.

To prove Proposition \ref{notyprop}, we will need two fairly straightforward lemmas, which will also be useful later. The first essentially says that the population at $y$ grows at least like $e^{\xi(y)t-2dt}$.

\begin{lem}\label{tr}
For any $y,z\in \Z^d$, $\theta>0$, and $0<s<t$,
\[\Ex_{z}\big[e^{\int_0^{s} \xi(X(u))du} \ind_{A(y,s;U_{y,\theta})}\big]\leq e^{-(\xi(y) - 2d)(t-s)}\Ex_{z}\big[e^{\int_0^{t} \xi(X(u))du} \ind_{A(y,t;U_{y,\theta})}\big].\]
\end{lem}

The second lemma says that if we start with one particle at $z$, the expected population at $y$ decreases as $z$ moves away from $y$. By reversing time (which is how we will apply this lemma later) we can think of starting with one particle at $y$, and the expected population decreasing as we look further away from $y$.

\begin{lem}\label{ybest}
For $y\in\Z^d$, let $D_k(y) = \{z\in \Z^d : |z-y|=k\}$. Suppose $\eta\in(0,1/(8d))$ and $\theta > 2d + (\log 2)/\eta$. Then for any $y\in \Z^d$, $k\geq1$ and $s>\eta$,
\[\sup_{z\in D_k(y)}\Ex_{z}\big[e^{\int_0^{s} \xi(X(u))du} \ind_{A(y,s;U_{y,\theta})}\big] < 8d\eta \sup_{z\in D_{k-1}(y)}\Ex_{z}\big[e^{\int_0^{s} \xi(X(u))du} \ind_{A(y,s;U_{y,\theta})}\big].\]
In particular,
\[\sup_{z\in \Z^d}\Ex_{z}\big[e^{\int_0^{s} \xi(X(u))du} \ind_{A(y,s;U_{y,\theta})}\big] = \Ex_{y}\big[e^{\int_0^{s} \xi(X(u))du} \ind_{A(y,s;U_{y,\theta})}\big]=f_\theta(y,s).\]
\end{lem}

These lemmas are not difficult to prove, but we first check that they imply Proposition \ref{notyprop}.

\begin{proof}[Proof of Proposition \ref{notyprop}]
By Markov's inequality,
\begin{multline*}\Px\bigg(\sum_{v\in L_\theta(y,t)}\sum_{z\neq y} N^v(z,t;U_{y,\theta}) \geq \sum_{v\in L_\theta(y,t)}\frac{f_\theta(y,t-\tau_y(v))}{\xi(y)^{9/10}} \hsl\bigg|\hsl\calG_{L_\theta(y,t)}\bigg)\\
\leq \frac{\sum_{v\in L_\theta(y,t)}\sum_{z\neq y} \Ex[ N^v(z,t;U_{y,\theta})|\calG_{L_\theta(y,t)}]}{\xi(y)^{-9/10}\sum_{v\in L_\theta(y,t)}f_\theta(y,t-\tau_y(v))}
\end{multline*}
almost surely. Also note that for any $v\in L_\theta(y,t)$,
\[\Ex[N^v (z,t;U_{y,\theta}) |\calG_{L_\theta(y,t)}] = \Ex_y[N(z,s;U_{y,\theta})]|_{s=t-\tau_y(v)}\]
so it suffices to show that
\[\sum_{z\neq y} \Ex_y[ N(z,s;U_{y,\theta})] \leq \frac{2d\xi(y)^{-19/20}}{(1-2^{-19/20})^2} f_\theta(y,s)\]
uniformly in $s>0$.

We apply Lemma \ref{ybest} with $\eta = 1/(8d\xi(y)^{19/20})$. (This is valid since by assumption $\theta > 10d\xi(y)^{19/20} > 2d + (\log2)/\eta$.) This tells us that for any $k\geq 1$,
\[\sup_{z\in D_k(y)} \Ex_z[e^{\int_0^s \xi(X(u))du}\ind_{A(y,s;U_{y,\theta})}] < \xi(y)^{-19k/20}\Ex_y[e^{\int_0^s \xi(X(u))du}\ind_{A(y,s;U_{y,\theta})}].\]
Therefore, summing over the $2kd$ vertices $z\in D_k(y)$ and then over $k$,
\begin{multline*}
\sum_{z\neq y} \Ex_z[e^{\int_0^s \xi(X(u))du}\ind_{A(y,s;U_{y,\theta})}]\\
\leq \sum_{k=1}^\infty 2kd \xi(y)^{-19k/20}\Ex_y[e^{\int_0^s \xi(X(u))du}\ind_{A(y,s;U_{y,\theta})}] \\
\leq \frac{2d\xi(y)^{-19/20}}{(1-2^{-19/20})^2} \Ex_y[e^{\int_0^s \xi(X(u))du}\ind_{A(y,s;U_{y,\theta})}].
\end{multline*}

Note that by time reversal and the many-to-one formula, for any $z\in\Z^d$
\[\Ex_z[e^{\int_0^s \xi(X(u))du}\ind_{A(y,s;U_{y,\theta})}] = \Ex_y[e^{\int_0^s \xi(X(u))du}\ind_{A(z,s;U_{y,\theta})}] = \Ex_y[ N(z,s;U_{y,\theta})],\]
and so we have shown that
\[\sum_{z\neq y} \Ex_y[ N(z,s;U_{y,\theta})] \leq \frac{2d\xi(y)^{-19/20}}{(1-2^{-19/20})^2} \Ex_y[ N(y,s;U_{y,\theta})]= \frac{2d\xi(y)^{-19/20}}{(1-2^{-19/20})^2}  f_\theta(y,s)\]
uniformly in $s$, as required.
\end{proof}

We now prove Lemmas \ref{tr} and \ref{ybest}.

\begin{proof}[Proof of Lemma \ref{tr}]
By time reversal,
\begin{align*}
\Ex_{z}\big[e^{\int_0^{t} \xi(X(u))du} \ind_{A(y,t;U_{y,\theta})}\big] &= \Ex_{y}\big[e^{\int_0^{t} \xi(X(u))du} \ind_{A(z,t;U_{y,\theta})}\big]\\
&\geq \Ex_{y}\big[e^{\int_0^{t} \xi(X(u))du} \ind_{A(z,t;U_{y,\theta})}\ind_{\{X(u)=y \hsl\forall u\leq t-s\}}\big].
\end{align*}
By the Markov property at time $t-s$, this is at least
\[e^{(\xi(y) - 2d)(t-s)}\Ex_{y}\big[e^{\int_0^{s} \xi(X(u))du} \ind_{A(z,s;U_{y,\theta})} \big].\]
Reversing time again we get the result.
\end{proof}

\begin{proof}[Proof of Lemma \ref{ybest}]
Let $J$ be the time of the first jump of our random walk, $J = \inf\{u>0 : X(u)\neq X(0)\}$, and let $B_k = \{z\in\Z^d : |z-y|\geq k\}$. Then
\begin{multline*}
\sup_{z\in B_k}\Ex_{z}\big[e^{\int_0^{s} \xi(X(u))du} \ind_{A(y,s;U_{y,\theta})}\big]\\
\leq\sup_{z\in B_k}\Ex_{z}\big[e^{\int_0^{s} \xi(X(u))du} \ind_{A(y,s;U_{y,\theta})}\ind_{\{J\leq \eta\}}\big]\\
+ \sup_{z\in B_k}\Ex_{z}\big[e^{\int_0^{s} \xi(X(u))du} \ind_{A(y,s;U_{y,\theta})}\ind_{\{J>\eta\}}\big].
\end{multline*}
If $J>\eta$ and $X(0)\neq y$, then on $A(y,s;U_{y,\theta})$ we have $\xi(X(u))\leq \xi(y)-\theta$ for all $u\leq \eta$. Therefore, by the Markov property,
\begin{multline*}
\sup_{z\in B_k}\Ex_{z}\big[e^{\int_0^{s} \xi(X(u))du} \ind_{A(y,s;U_{y,\theta})}\big]\\
\leq e^{\xi(y)\eta}\Px(J\leq \eta)\sup_{z\in B_{k-1}}\Ex_{z}\big[e^{\int_0^{s-\eta} \xi(X(u))du} \ind_{A(y,s-\eta;U_{y,\theta})}\big]\\
+ e^{(\xi(y)-\theta)\eta}\sup_{z\in B_k}\Ex_{z}\big[e^{\int_0^{s-\eta} \xi(X(u))du} \ind_{A(y,s-\eta;U_{y,\theta})}\big].
\end{multline*}
We have $\Px(J\leq \eta) = 1-e^{-2d\eta} \leq 2d\eta$, so applying Lemma \ref{tr},
\begin{multline*}
\sup_{z\in B_k}\Ex_{z}\big[e^{\int_0^{s} \xi(X(u))du} \ind_{A(y,s;U_{y,\theta})}\big]\\
\leq 2d\eta e^{2d\eta}\sup_{z\in B_{k-1}}\Ex_{z}\big[e^{\int_0^{s} \xi(X(u))du} \ind_{A(y,s;U_{y,\theta})}\big]\\
+ e^{-\theta\eta + 2d\eta}\sup_{z\in B_k}\Ex_{z}\big[e^{\int_0^{s} \xi(X(u))du} \ind_{A(y,s;U_{y,\theta})}\big].
\end{multline*}
Rearranging, we obtain
\[\sup_{z\in B_k}\Ex_{z}\big[e^{\int_0^{s} \xi(X(u))du} \ind_{A(y,s;U_{y,\theta})}\big]\leq \frac{2d\eta e^{2d\eta}}{1-e^{(2d-\theta)\eta}}\sup_{z\in B_{k-1}}\Ex_{z}\big[e^{\int_0^{s} \xi(X(u))du} \ind_{A(y,s;U_{y,\theta})}\big].\]
Note that if $\eta<1/(8d)$ and $\theta > 2d + (\log 2)/\eta$, then $\frac{2d\eta e^{2d\eta}}{1-e^{(2d-\theta)\eta}} < 8d\eta < 1$. Therefore either the left-hand side above is zero --- in which case the result trivially holds --- or the supremum on the right-hand side must be attained at some point in $B_{k-1}\setminus B_k = D_{k-1}(y)$. The first statement in the lemma follows by induction on $k$. The second statement follows from the first together with, for the last equality, the many-to-one formula.
\end{proof}

\section{The population at $y$ is big enough}\label{ysec}

Our main aim in this section is to prove the following result.

\begin{prop}\label{bigenough}
Suppose that $y\in\Z^d$ satisfies $(3+d)\xi(y)^{-1/16}\leq 1/2$ and $\xi(y)\geq e^{\sqrt{256+100d}}$. Suppose also that $\theta \geq 2d +\frac{\log 2}{16d}$. Then, if either $t  - H(y) \geq \frac{1 + \frac 1 4 \log \xi(y)}{\xi(y)}$ or $\# L_\theta(y,t) \geq \xi(y)^{1/2}$, we have
\begin{multline*}
\Px\Big(\sum_{v\in L_\theta(y,t)} N(y,t;U_{y,\theta}) < \xi(y)^{-4/5}\sum_{v\in L_\theta(y,t)} f_\theta(y,t-\tau_y(v))\Big| \calG_{L_\theta(y,t)}\Big)\\
\leq (3+d)\xi(y)^{-1/16} + \xi(y)te^{-\frac{1}{16}\xi(y)^{1/2}}.
\end{multline*}
\end{prop}

The idea is simple: each particle that hits $y$ for the first time gives rise to at least $c f_\theta(y,t-\tau_y(v))$ particles at $y$ at time $t$, for some small constant $c$, with high probability. The conditions on $\xi(y)$ can be read as ``$\xi(y)$ is large'' (which will be true in the case we are interested in, since we will apply this result to the optimal point in the whole of $\Z^d$). Unfortunately the details are quite intricate, since there could be large fluctuations in how many particles are hitting $y$ at different times. We proceed via a series of lemmas.

\begin{lem}\label{pz}
For any $y\in\Z^d$ with $\xi(y)\geq 8d$, any $\theta \geq 2d + (\log 2)/16d$, and any $t\geq 1/\xi(y)$,
\[\Px_{y}\big( N(y,t;U_{y,\theta}) \geq \smfr{1}{2}\Ex_y[N(y,t;U_{y,\theta})]\big) \geq 1/16.\]
\end{lem}

\begin{proof}
This is a relatively straightforward second moment calculation. Fix $t$ and $y$. We use the many-to-two formula \cite[Section 4.2]{HR11}. This says that
\begin{multline*}
\Ex_y[N(y,t;U_{y,\theta})^2] - \Ex_y[N(y,t;U_{y,\theta})]\\
= \int_0^{t} \Ex_y\Big[\xi(X^{1,s}(s)) e^{\int_0^{t} \xi(X^{1,s}(u)) du + \int_s^{t} \xi(X^{2,s}(u)) du} \ind_{A^{1,s}\cap A^{2,s}}\Big] ds
\end{multline*}
where for each $s\geq0$, $(X^{1,s}(u))_{u\geq0}$ and $(X^{2,s}(u))_{u\geq0}$ satisfy
\begin{itemize}
\item $(X^{i,s}(u))_{u\geq0}$ is a simple symmetric continuous-time random walk on $\Z^d$, jumping to each neighbouring vertex at rate 1, for each $i=1,2$;
\item $X^{1,s}(u) = X^{2,s}(u)$ for all $u\leq s$;
\item $(X^{1,s}(s+u) - X^{1,s}(s))_{u\geq0}$ and $(X^{2,s}(s+u) - X^{2,s}(s))_{u\geq0}$ are independent,
\end{itemize}
and for each $i=1,2$, $A^{i,s}$ is the event
\[A^{i,s} = \{ X^{i,s}(t)=y, \hs \not\exists u\leq t \hbox{ with } X^{i,s}(u)\in U_{y,\theta}\}.\]

Note that on $A^{1,s}$, we have $\xi(X^{1,s}(u)) \leq \xi(y)$ for all $u\leq s\leq t$, so
\begin{multline*}
\Ex_y\Big[\xi(X^{1,s}(s)) e^{\int_0^{t} \xi(X^{1,s}(u)) du + \int_s^{t} \xi(X^{2,s}(u)) du} \ind_{A^{1,s}\cap A^{2,s}}\Big]\\
\leq \Ex_y\Big[\xi(y) e^{\xi(y)s + \int_s^{t} \xi(X^{1,s}(u)) du + \int_s^{t} \xi(X^{2,s}(u)) du} \ind_{A^{1,s}\cap A^{2,s}}\Big].
\end{multline*}
Also, since $(X^{1,s}(s+u) - X^{1,s}(s))_{u\geq0}$ and $(X^{2,s}(s+u) - X^{2,s}(s))_{u\geq0}$ are independent, by the Markov property the right-hand side above is at most
\[\xi(y) e^{\xi(y)s} \sup_{z\in\Z^d}\Ex_z[e^{\int_0^{t-s}\xi(X(u))du}\ind_{A(y,t-s;U_{y,\theta})}]^2.\]
Thus
\begin{multline*}\Ex_y[ N(y,t;U_{y,\theta})^2] - \Ex_y[N(y,t;U_{y,\theta})]\\
\leq \int_0^{t} \xi(y) e^{\xi(y)s} \sup_{z\in\Z^d}\Ex_z[e^{\int_0^{t-s}\xi(X(u))du}\ind_{A(y,t-s;U_{y,\theta})}]^2 ds.
\end{multline*}
By Lemma \ref{tr}, this is at most
\[\int_0^{t} \xi(y) e^{(4d-\xi(y))s} \sup_{z\in\Z^d}\Ex_z[e^{\int_0^{t}\xi(X(u))du}\ind_{A(y,t;U_{y,\theta})}]^2 ds\]
and by Lemma \ref{ybest}, provided $\theta \geq 2d + (\log 2)/16d$, the supremum is achieved at $y$, so we get
\begin{multline*}
\Ex_y[ N(y,t;U_{y,\theta})^2] - \Ex_y[ N(y,t;U_{y,\theta})]\\
\leq \int_0^{t} \xi(y) e^{(4d-\xi(y))s} \Ex_y[e^{\int_0^{t}\xi(X(u))du}\ind_{A(y,t;U_{y,\theta})}]^2 ds\\
\leq \frac{\xi(y)}{\xi(y)-4d} \Ex_y[ N(y,t;U_{y,\theta}) ]^2.
\end{multline*}
To finish off, note that by Lemma \ref{tr}, when $\xi(y)\geq 8d$ we have
\[\Ex_y[ N(y,t;U_{y,\theta})] \le \Ex_y[ N(y,t;U_{y,\theta})] e^{(\xi(y) - 2d)t} \le \Ex_y[N(y,t;U_{y,\theta})]^2;\]
combining this with the estimate above we get
\[\Ex_y[ N(y,t;U_{y,\theta})^2] \leq 4 \Ex_y[ N(y,t;U_{y,\theta})]^2\]
and the result follows from the Paley-Zygmund inequalty.
\end{proof}

\begin{lem}\label{onev}
Suppose that $y\in\Z^d$ satisfies $\xi(y)\geq 8d$, and that $\theta \geq 2d + \frac{\log 2}{16d}$ and $t\geq \frac{1}{\xi(y)} + \frac{\log\xi(y)}{4\xi(y)}$. Then
\[\Px_y\Big(N(y,t;U_{y,\theta}) \leq \smfr{1}{2\xi(y)^{1/4}} f_\theta(y,t)\Big) \leq (3+d) \xi(y)^{-1/16}.\]
\end{lem}

\begin{proof}
For $s\geq0$, let
\[\Upsilon_s = \#\{v\in Y(y,s) : X_v(u) = y \hs \forall u\leq s\}.\]
Then $(\Upsilon_s)_{s\geq0}$ is a birth-death process with birth rate $\xi(y)$ and death rate $2d$.

Let $D_1$ be the time of the first death (i.e.~the first time a particle leaves $y$), and let $T_n = \inf\{s\geq0 : \Upsilon_s = n\}$. If $D_1\geq T_n$ then $T_n$ is simply the time of the $(n-1)$th birth. Therefore, for any $u\geq0$,
\[\Px_y(u\leq T_n \leq D_1) \leq P\bigg(\sum_{j=1}^{n-1} V_j \geq u\bigg)\]
where under $P$, the random variables $(V_j)_{j\geq1}$ are independent, and $V_j$ is exponentially distributed with parameter $\xi(y)j$ for each $j$. Thus
\[\Px_y(u\leq T_n \leq D_1) \leq E[e^{\frac{\xi(y)}{2}\sum_{j=1}^{n-1} V_j}]e^{-\xi(y)u/2} \leq \prod_{j=1}^{n-1} (1-\smfr{1}{2j})^{-1} e^{-\xi(y)u/2}.\]
But
\[\prod_{j=1}^{n-1} \left(1-\smfr{1}{2j}\right)^{-1} = \exp\bigg(-\sum_{j=1}^{n-1} \log\left(1-\smfr{1}{2j}\right)\bigg) \leq \exp\bigg(\sum_{j=1}^{n-1} \frac{1}{j}\bigg) \leq n,\]
so fixing $n = \lceil \exp(\xi(y)u/4) \rceil$ we get
\[\Px_y(u\leq T_n \leq D_1) \leq 2 e^{-\xi(y)u/4}.\]
But
\[\Px_y(D_1 < u\wedge T_n) \leq 1-e^{-2dnu} \leq 2dnu,\]
so fixing $u = \frac{\log \xi(y)}{4\xi(y)}$, we have
\[\Px_y( T_n > D_1\wedge u ) \leq 2dnu + 2 e^{-\xi(y)u/4} \leq (2+d)\xi(y)^{-1/16}.\]
We now concentrate on the event $\{T_n\leq D_1\wedge u\}$.

On the event $\{T_n\leq D_1\wedge u\}$, at time $T_n$ we have $n$ particles at $y$ (that have never left $y$). Each of these has an independent descendance, and therefore (almost surely on the event $\{T_n\leq D_1\wedge u\}$)
\[\Px_y\Big(N(y,t;U_{y,\theta}) \leq \smfr{1}{2\xi(y)^{1/4}}f_\theta(y,t) \Big| \F_{T_n}\Big) \leq \sup_{s\in[0,u]}\Px_y\Big( N(y,t-s;U_{y,\theta}) \leq \smfr{1}{2\xi(y)^{1/4}}f_\theta(y,t)\Big)^n.\]
Clearly $f_\theta(y,t)\leq e^{\xi(y)s}f_\theta(y,t-s)$ for any $s\in[0,t]$, and thus by our choice of $u$, on the event $\{T_n\leq D_1\wedge u\}$,
\[\Px_y\Big( N(y,t;U_{y,\theta}) \leq \smfr{1}{2\xi(y)^{1/4}}f_\theta(y,t) \Big| \F_{T_n}\Big) \leq \sup_{s\in[0,u]}\Px_y\Big( N(y,t-s;U_{y,\theta}) \leq \smfr{1}{2}f_\theta(y,t-s)\Big)^n.\]
Applying Lemma \ref{pz} tells us that this is at most $(15/16)^n$, which is smaller than $\xi(y)^{-1/16}$, as required.
\end{proof}

For each $j\geq0$, define
\[I_j(z) = [H(z) + j/\xi(z), H(z) + (j+1)/\xi(z))\]
and
\[L_{j,\theta}(z,t) = \{v\in L_\theta(z,t) : \tau_z(v) \in I_j(z)\}.\]

\begin{lem}\label{jbulk}
Suppose that $y\in\Z^d$ satisfies $\xi(y)\geq 8d$ and $(3+d) \xi(y)^{-1/16} \leq 1/2$, and that $\theta \geq 2d + \frac{\log 2}{16d}$. If $\#L_{j,\theta}(y,t) \geq \xi(y)^{1/2}$ and $t\geq H(y) + \frac{j+2}{\xi(y)} + \frac{\log \xi(y)}{4\xi(y)}$,  then
\[\Px\Big(\sum_{v\in L_{j,\theta}(y,t)} N^v(y,t;U_{y,\theta}) < \frac{1}{8e\xi(y)^{1/4}}\sum_{v\in L_{j,\theta}(y,t)} f_\theta(y,t-\tau_y(v)) \Big| \calG_{L_\theta(y,t)}\Big)\leq e^{-\frac{1}{16}\xi(y)^{1/2}}.\]
\end{lem}

\begin{proof}
We recall the Chernoff bound from \cite[Lemma 2.6]{OR14}, which said that if $Z_1,\ldots,Z_k$ are independent Bernoulli random variables and $Z = \sum_{i=1}^k Z_i$, then
\[P\left(Z\leq \frac{E[Z]}{2} \right)\leq \exp\left(-\frac{E[Z]}{8} \right).\]
Lemma \ref{onev} tells us that for any $v\in L_{j,\theta}(y,t)$,
\[\Px\Big(N^v(y,t;U_{y,\theta}) \leq \smfr{1}{2\xi(y)^{1/4}} f_\theta(y,t-\tau_y(v))\Big|\calG_{L_\theta(y,t)}\Big) \leq (3+d) \xi(y)^{-1/16} \leq 1/2.\]
Letting $Z_v$ be the indicator that $N^v(y,t;U_{y,\theta}) > \smfr{1}{2\xi(y)^{1/4}} f_\theta(y,t-\tau_y(v))$, we get
\begin{equation}\label{vchern}
\Px\big(\#\{v\in L_{j,\theta}(y,t) : Z_v = 1\} \leq \smfr{1}{4} \# L_{j,\theta}(y,t) \big| \calG_{L_\theta(y,t)}\big) \leq e^{-\frac{1}{16}\#L_{j,\theta}(y,t)} \leq e^{-\frac{1}{16}\xi(y)^{1/2}}.
\end{equation}

Note that if $v\in L_{j,\theta}(y,t)$, then by Lemma~\ref{tr}
\[f_\theta(y,t-\tau_y(v)) \geq f_\theta(y,t-H(y)-\smfr{j+1}{\xi(y)}) \geq e^{-1} f_\theta(y,t-H(y)-\smfr{j}{\xi(y)}).\]
Therefore (\ref{vchern}) tells us that
\[\Px\Big(\sum_{v\in L_{j,\theta}(y,t)} \hspace{-3mm} N^v(y,t;U_{y,\theta}) \leq \frac{1}{8e\xi(y)^{1/4}}\sum_{v\in L_{j,\theta}(y,t)} \hspace{-3mm} f_\theta(y,t-H(y)-\smfr{j}{\xi(y)})\Big|\calG_{L_\theta(y,t)}\Big) \leq e^{-\frac{1}{16}\xi(y)^{1/2}}\]
and the result follows from the fact that for any $v\in L_{j,\theta}(y,t)$, since $H(y)+j/\xi(y)\le \tau_y(v)$, we have $f_\theta(y,t-H(y)-\smfr{j}{\xi(y)}) \geq f_\theta(y,t-\tau_y(v))$.
\end{proof}

Lemma \ref{jbulk} takes care of most values of $j$. However we still need to consider some ``boundary'' cases, when either the number of particles absorbed in $I_j(y)$ is small, or $t-\tau_y(v)$ is small. We begin with the latter. Let
\[\bar L_\theta(y,t) = \{v\in L_\theta(y,t) : t-\tau_y(v) \in [0,\smfr{2}{\xi(y)} + \smfr{\log \xi(y)}{4\xi(y)}]\}.\]

\begin{lem}\label{tsmall}
Suppose that $y\in\Z^d$ satisfies $\xi(y)\geq e^{\sqrt{256+100d}}$, and that $\theta\geq 0$. If $\#\bar L_\theta(y,t) \geq \xi(y)^{1/2}$, then
\[\Px\Big(\sum_{v\in \bar L_\theta(y,t)} N^v(y,t;U_{y,\theta}) < \frac{1}{4e^2\xi(y)^{1/4}}\sum_{v\in \bar L_\theta(y,t)} f_\theta(y,t-\tau_y(v)) \Big| \calG_{L_\theta(y,t)}\Big)\leq e^{-\frac{1}{16}\xi(y)^{1/2}}.\]
\end{lem}

\begin{proof}
The argument is, at heart, very simple: if $v\in \bar L_\theta(y,t)$, then since $t-\tau_y(v)$ is so small, there is not enough time for any of $v$'s descendants to leave $y$ before $t$. We now provide the details.

Recall the notation from Lemma \ref{onev}: when starting with one particle at $y$, $\Upsilon_s$ is the number of particles that have stayed at $y$ up to time $s$, $D_1$ is the first time a particle leaves $y$, and $T_n$ is the first time there are $n$ particles at $y$ who have never left.  Let $u=\smfr{2}{\xi(y)} + \smfr{\log \xi(y)}{4\xi(y)}$ and $n=\lceil \xi(y)^{1/2}\rceil$. Then
\[\Px\big( N^v(y,t;U_{y,\theta}) < 1 \big| \calG_{L_\theta(y,t)}\big)\leq \Px_y(\Upsilon_u = 0).\]
Also, if $V_j$ is exponentially distributed with parameter $\xi(y)j$ and $(V_j)_{j\geq1}$ are independent, setting $V = \sum_{i=1}^{n-1} V_j$,
\[\Px_y(T_n<u) \leq P(V < u) \leq P\big(|V-E[V]| > E[V]-u) \leq \frac{E[(V-E[V])^2]}{(E[V]-u)^2}.\]
But
\[E[(V-E[V])^2] = \sum_{i=1}^{n-1} \hbox{Var}(V_i) = \sum_{i=1}^{n-1} \frac{1}{\xi(y)^2 i^2} \leq \frac{2}{\xi(y)^2},\]
and
\[E[V]-u = \sum_{i=1}^{n-1} \frac{1}{\xi(y)i} - u > \frac{\log n}{\xi(y)} - \frac{2}{\xi(y)} - \frac{\log \xi(y)}{4\xi(y)} \geq \frac{1}{\xi(y)}(\smfr{1}{4}\log\xi(y)-2) \geq \frac{\log \xi(y)}{8\xi(y)}\]
since $\xi(y)\geq e^{16}$. Thus $\Px_y(T_n<u) \leq 128/(\log \xi(y))^2$.

Then
\begin{multline*}
\Px_y(\Upsilon_u=0) \leq \Px_y(D_1<u) \leq \Px_y(T_n<u) + \Px_y(D_1 < u\wedge T_n)\\
\leq \frac{128}{(\log \xi(y))^2} + 1-e^{-2dnu} \leq \frac{128}{(\log \xi(y))^2} + 2dnu \leq \frac{128 + 50d}{(\log\xi(y))^2} \leq \frac12
\end{multline*}
where we used the fact that $nu\leq 25/\log^2 \xi(y)$, since $(\log x)^3 / x^{1/2}\leq 25$ for all $x>1$.
Applying the Chernoff bound (\cite[Lemma 2.6]{OR14}) in the same way as in the proof of Lemma \ref{jbulk}, we get
\begin{align*}\Px\big(\#\{v\in \bar L_\theta(y,t) : N^v(y,t;U_{y,\theta}) \ge 1\} < \smfr{1}{4}\#\bar L_\theta(y,t) \big|\calG_{L_\theta(y,t)} \big) &\leq e^{-\frac{1}{16}\#\bar L_\theta(y,t)}\\
&\leq e^{-\frac{1}{16}\xi(y)^{1/2}}.
\end{align*}
For any $v\in\bar L_\theta(y,t)$ we have $t-\tau_y(v)\le \frac{2}{\xi(y)} + \frac{\log \xi(y)}{4\xi(y)}$, so
\[f_\theta(y,t-\tau_y(v))\le e^{\xi(y)(\frac{2}{\xi(y)} + \frac{\log \xi(y)}{4\xi(y)})}f_\theta(y,0) = e^2 \xi(y)^{1/4}.\]
Therefore
\[\Px\Big(\sum_{v\in \bar L_\theta(y,t)} N^v(y,t;U_{y,\theta}) < \frac{1}{4e^2\xi(y)^{1/4}}\sum_{v\in \bar L_\theta(y,t)}f_\theta(y,t-\tau_y(v)) \Big|\calG_{L_\theta(y,t)}\Big) \leq e^{-\frac{1}{16}\xi(y)^{1/2}}.\qedhere\]
\end{proof}

Setting
\[J(y) = \{j\in\{0,\ldots,\lfloor \xi(y)(t-H(y))\rfloor\} : \#L_{j,\theta}(y,t)\geq \xi(y)^{1/2}\},\]
we can combine Lemmas \ref{jbulk} and \ref{tsmall} to get the following corollary.

\begin{cor}\label{bigcor}
Suppose that $y\in\Z^d$ satisfies $(3+d) \xi(y)^{-1/16} \leq 1/2$ and $\xi(y)\geq e^{\sqrt{256+100d}}$, and that $\theta \geq 2d + \frac{\log 2}{16d}$. Then
\begin{multline*}
\Px\Big(\sum_{j\in J(y)} \sum_{v\in L_{j,\theta}(y,t)} N^v(y,t;U_{y,\theta}) \leq \frac{1}{4e^2\xi(y)^{1/4}} \sum_{j\in J(y)}\sum_{v\in L_{j,\theta}(y,t)} f_\theta(y,t-\tau_y(v)) \Big| \calG_{L_\theta(y,t)}\Big)\\
\leq \xi(y) t e^{-\frac{1}{16}\xi(y)^{1/2}}.
\end{multline*}
\end{cor}

\begin{proof}
For any $j$, either $t\ge H(y)+\frac{j+2}{\xi(y)} + \frac{\log \xi(y)}{4\xi(y)}$, in which case we apply Lemma \ref{jbulk}, or $t< H(y) + \frac{j+2}{\xi(y)} + \frac{\log \xi(y)}{4\xi(y)}$, in which case for any $v\in L_{j,\theta}(y,t)$ we have $t<\tau_y(v)+\frac{2}{\xi(y)} + \frac{\log \xi(y)}{4\xi(y)}$ and we apply Lemma \ref{tsmall}. Since $8e<4e^2$, we get
\begin{multline*}
\Px\Big(\exists j\in J(y) : \sum_{v\in L_{j,\theta}(y,t)} N^v(y,t;U_{y,\theta}) \leq \frac{1}{4e^2\xi(y)^{1/4}} \sum_{v\in L_{j,\theta}(y,t)} f_\theta(y,t-\tau_y(v)) \Big| \calG_{L_\theta(y,t)}\Big)\\
\leq \xi(y) t e^{-\frac{1}{16}\xi(y)^{1/2}},
\end{multline*}
and the result follows.
\end{proof}

We now look after those $j$ for which the number of particles absorbed in $I_j(y)$ is small.

\begin{lem}\label{smalllem}
If $y\in \Z^d$ satisfies $\xi(y)\geq 4d$, then
\[\sum_{j\not\in J(y)} \sum_{v\in L_{j,\theta}(y,t)} f_\theta(y,t-\tau_y(v)) \leq
3\xi(y)^{1/2} f_\theta(y,t-H(y)).\]
\end{lem}

\begin{proof}
Note that by Lemma \ref{tr}, for $v \in L_{j,\theta}(y,t)$,
\[f_\theta(y, t- \tau_y(v)) \leq
f_\theta(y,t-H(y)-\tfrac{j}{\xi(y)}) \leq e^{-(1-\frac{2d}{\xi(y)})j} f_\theta(y,t-H(y))
\leq e^{-j/2} f_\theta(y,t - H(y)) .
\]
Thus, we obtain
\[ \begin{aligned} \sum_{j\not\in J(y)} \sum_{v\in L_{j,\theta}(y,t)} f_\theta(y,t-\tau_y(v))
&  \leq \sum_{j \not\in J(y)}  \# L_{j,\theta}(y,t) e^{-j/2} f_\theta(y,t-H(y)) \\
& \leq \frac{1}{1-e^{-1/2}}
 \xi(y)^{1/2} f_\theta(y,t- H(y)) ,
 \end{aligned}
\]
which implies the result as $\frac{1}{1-e^{-1/2}} \leq 3$.
\end{proof}

We are now in a position to prove Proposition \ref{bigenough}. The strategy is the following. We have shown in Corollary \ref{bigcor} that the number of particles ``coming from $J(y)$'' cannot be much smaller than the expected number of particles ``coming from $J(y)$''. We would also like to say that the number of particles ``not from $J(y)$'' can't be much less than the expected number of particles ``not from $J(y)$''. But this is difficult, because in a ``non-$J(y)$'' time interval we might---for example---have only one particle that just happens to hit $y$ very briefly and then run off before it has time to breed.

We instead show that the number of particles descended from $v_0$, defined to be the first particle to hit $y$,
can't be much less than the expected number of particles "not from $J(y)$". To this end, we have shown that the expected number of particles ``not from $J(y)$'' is not much bigger than the expected number of particles descended from $v_0$ (Lemma \ref{smalllem}); and that the number of particles descended from $v_0$ is not much smaller than the expected number of particles descended from $v_0$ (Lemma \ref{onev}). Putting these elements together gives us a proof.

\begin{proof}[Proof of Proposition \ref{bigenough}]
Suppose first that $t - H(y) \in [0, \frac{1 + \frac 1 4 \log\xi(y)}{\xi(y)}]$, but $\# L_\theta(y,t) \geq \xi(y)^{1/2}$. 
In this case $L_\theta(y,t) = \bar L_\theta(y,t)$ and Proposition~\ref{bigenough} follows from Lemma~\ref{tsmall}.

Therefore we can assume that $t - H(y) \geq \frac{1 + \frac 1 4 \log\xi(y)}{\xi(y)}$. 
We let $v_0$ be the first particle to hit $y$ (at time $H(y)$). We want to show that except on a set of small probability, we have
\[\sum_{v\in L_\theta(y,t)} N(y,t;U_{y,\theta}) \geq \xi(y)^{-4/5} \sum_{v\in L_{\theta}(y,t)} f_\theta(y,t-\tau_y(v)).\]

Since
\[\sum_{v\in L_\theta(y,t)} N^v(y,t;U_{y,\theta}) \geq \frac12 \sum_{j\in J}\sum_{v\in L_{j,\theta}(y,t)} N^v(y,t;U_{y,\theta}) + \frac12 N^{v_0}(y,t;U_{y,\theta}),\]
and by Lemma \ref{smalllem},
\[\sum_{v\in L_\theta(y,t)} f_\theta(y,t-\tau_y(v)) \leq \sum_{j\in J} \sum_{v\in L_{j,\theta}(y,t)} f_\theta(y,t-\tau_y(v)) + 3 \xi(y)^{1/2} f_\theta(y,t-H(y)),\]
it suffices to show that
\begin{multline*}
\frac12 \sum_{j\in J}\sum_{v\in L_{j,\theta}(y,t)} N^v(y,t;U_{y,\theta}) + \frac12 N^{v_0}(y,t;U_{y,\theta})\\
\geq \xi(y)^{-4/5}\sum_{j\in J} \sum_{v\in L_{j,\theta}(y,t)} f_\theta(y,t-\tau_y(v)) + 3 \xi(y)^{-3/10} f_\theta(y,t-H(y)).
\end{multline*}

By Corollary \ref{bigcor} (noting that $8e^2\xi(y)^{1/4} \leq \xi(y)^{4/5}$),
\begin{multline*}
\Px\Big(\frac12 \sum_{j\in J(y)}\sum_{v\in L_{j,\theta}(y,t)} N^v(y,t;U_{y,\theta}) < \xi(y)^{-4/5} \sum_{j\in J(y)} \sum_{v\in L_{j,\theta}(y,t)} f_\theta(y,t-\tau_y(v))\Big|\calG_{L_\theta(y,t)}\Big) \\
\leq \xi(y)te^{-\frac{1}{16}\xi(y)^{1/2}},
\end{multline*}
and by Lemma \ref{onev} (using that $t - H(y) \geq \frac{1 + \frac 1 4 \log\xi(y)}{\xi(y)}$ and $3\xi(y)^{-3/10}\leq \frac{1}{2\xi(y)^{1/4}}$), we get
\[\Px\Big(\frac12 N^{v_0}(y,t;U_{y,\theta}) < 3\xi(y)^{-3/10} f_\theta(y,t-H(y))\Big|\calG_{L_\theta(y,t)}\Big) \leq (3+d)\xi(y)^{-1/16}.\]
The result follows.
\end{proof}

\section{Applying bounds at $w_T(t)$} \label{mainproofsec}

We essentially have everything we need to prove one-point localization, by combining the results from the previous two sections with our previous work from~\cite{OR14}.
We therefore need to recall some of the notation from~\cite{OR14}.
We introduce a rescaling
of time by a parameter $T > 0$. We also rescale space 
and the potential. 
If  $q=\frac{d}{\alpha-d}$, the right scaling factors for the potential, respectively space, 
turn out to be
\[ a(T) = \left(\frac{T}{\log T}\right)^q \quad \mbox{and} \quad r(T) = \left(\frac{T}{\log T}\right)^{q+1} . \]
We now define the rescaled lattice as
\[ L_T = \{ z \in \R^d \, : \, r(T) z \in \Z^d \}, \]
and for $z\in\R^d$, $R\geq 0$ define $L_T(z,R) = L_T\cap B(z,R)$ where $B(z,R)$ is the open ball of radius $R$ about $z$ in $\R^d$. For $z \in L_T$, the rescaled potential is given by 
\[ \xi_T(z) = \frac{\xi(r(T) z)}{a(T)},  \]
and we set $\xi_T(z) = 0$ for $z \in \R^d \setminus L_T$.

For any site $z \in L_T$, we set
\[ H_T(z) = \inf\{ t>0 : N(r(T)z,tT)\ge 1\}\]
and
\[ h_T(z) = \inf_{ \substack{y_0,\ldots,y_n \in L_T:\\ y_0 = z, y_n = 0}} \Bigg( \sum_{j=1}^{n} q\frac{| y_{j-1} - y_{j}|}{\xi_T(y_{j})} \Bigg);\]
we showed in~\cite{OR14} that these two quantities are close in a suitable sense. We also set, for $z\in L_T$ and $t\ge0$,
\[M_T(z,t) = \frac{1}{a(T)T} \log_+ N(r(T)z,tT)\]
and
\[m_T(z,t) = \sup_{y \in L_T}\{\xi_T(y)(t-h_T(y))_+ - q|z-y|\}.\]
Again we showed that these two quantities are close.

In order to apply our results from earlier sections, we need to ensure that several irritating events do not occur. We check, via a sequence of lemmas, that these events have small probability. All these lemmas are either easy to prove, or are restatements of results from \cite{OR14}. First we fix some parameters:
\begin{itemize}
\item $\rho_T = \log\log T$;
\item $\nu_T = \log^{-d/16\alpha} T$;
\item $K_T = \nu_T^{-2\alpha}\rho_T^{2d} = (\log T)^{d/16}(\log\log T)^{2d}$;
\item $\eps_T = \frac{3}{q}r(T)\log^{-1/4} T$.
\item $\theta_T = \nu_T^{2+2\alpha}a(T)$.
\end{itemize}

\begin{lem}
Define
\begin{multline*}
A_T = \{\exists z_1,z_2\in L_T(0,\rho_T) : z_1\neq z_2,\,\, \xi_T(z_1)\geq \nu_T/2,\\
 \xi_T(z_2)\geq \nu_T/2,\,\, |\xi_T(z_1)-\xi_T(z_2)|\leq \nu_T^{2+2\alpha}\}.
 \end{multline*}
Then $\P(A_T)\to 0$ as $T\to\infty$.
\end{lem}

\begin{proof}
Since there are at most $C_d^2 r(T)^{2d}\rho_T^{2d}$ pairs of points in $L_T(0,\rho_T)$, for any pair $z_1\neq z_2$ we have
\begin{multline*}
\P(A_T)\leq C_d^2 r(T)^{2d}\rho_T^{2d} \P(\xi_T(z_1)\geq \nu_T/2,\, \xi_T(z_2)\geq\nu_T/2,\\
\xi_T(z_2)\in[\xi_T(z_1)-\nu_T^{2+2\alpha},\xi_T(z_1)+\nu_T^{2+2\alpha}]).
\end{multline*}
Now, $\xi_T(z_1)$ and $\xi_T(z_2)$ are independent, and $\P(\xi_T(z_2)\in[x,x+y])$ is decreasing in $x$ for fixed $y$, so
\begin{align*}
\P(A_T)&\leq C_d^2 r(T)^{2d}\rho_T^{2d} \P(\xi_T(z_1)\geq \nu_T/2, \xi_T(z_2)\in[\nu_T/2,\nu_T/2+2\nu_T^{2+2\alpha}])\\
&\leq C_d^2 r(T)^{2d}\rho_T^{2d} 2^{2\alpha} a(T)^{-2\alpha}\nu_T^{-2\alpha}(1-(1+2\nu_T^{1+2\alpha})^{-\alpha})\\
&\leq C_d^2 \rho_T^{2d}\nu_T^{-2\alpha}2^{2\alpha+1}\alpha\nu_T^{1+2\alpha} = 2^{2\alpha+1}\alpha C_d^2 \rho_T^{2d}\nu_T
\end{align*}
which tends to $0$ as $T\to\infty$.
\end{proof}

\begin{lem}
Let $Q = \{z\in \Z^d : \xi(z)\leq \nu_T a(T)/2\}$. Then
\[\Px\bigg(\sum_{z\in\Z^d} N(z,tT; Q^c) > e^{\frac34 \nu_T a(T)tT}\bigg)\to 0\]
$\P$-almost surely as $T\to\infty$.
\end{lem}

\begin{proof}
By the many-to-one lemma,
\[\Ex\Big[\sum_{z\in\Z^d} N(z,tT; Q^c)\Big] \leq e^{\nu_T a(T)/2}\]
and the result follows from Markov's inequality.
\end{proof}

\begin{lem}\label{whaslargexi}
\[\P(\exists z_0\in L_T(0,\nu_T) \hbox{ with } \xi_T(z_0)>\nu_T \hbox{ and } m_T(z_0,t)\geq \smfr{9}{10}t\nu_T) \to 1\]
as $T\to\infty$. As a result,
\[\P\bigg(\sum_{z\in\Z^d}N(z,tT) < e^{\frac45 a(T)Tt\nu_T}\bigg)\to 0\]
as $T\to\infty$, and if $w\in L_T$ satisfies $m_T(w,t)\geq m_T(z,t)$ for all $z\in L_T$, then
\[\P( \xi_T(w)\leq \nu_T/2 ) \to 0\]
as $T\to\infty$.
\end{lem}

\begin{proof}
By \cite[Lemma 2.7(i)]{OR14}, we have
\[\P(\xi_T(z)\leq \nu_T \hs\forall z\in L_T(0,\nu_T)) \leq e^{-c_d \nu_T^{d-\alpha}}\]
which tends to $0$ as $T\to\infty$ since $\nu_T \rightarrow 0$ and $d-\alpha<0$. Therefore, with high probability, there exists some point $z_0\in L_T(0,\nu_T)$ with $\xi_T(z_0)>\nu_T$. But by \cite[Lemma 3.4]{OR14},
\[\P(\exists z\in L_T(0,\nu_T) : h_T(z) > \frac{4q}{1-2^{-1/(2q+2)}}\nu_T^{1/(2q+2)})\to 0,\]
so in particular, with high probability $h_T(z_0)\leq \frac{4q}{1-2^{-1/(2q+2)}}\nu_T^{1/(2q+2)}$ which is less than $t/10$ for large $T$. Since
\[m_T(z_0,t) \geq \xi_T(z_0)(t-h_T(z_0)),\]
the first part of the result follows. 

By \cite[Proposition 5.7]{OR14}, with high probability we have $M_T(z_0,t) \geq \frac{9}{10}t\nu_T  - \log^{-1/4} T$, and since $\nu_T = \log^{-d/16\alpha} T \geq \log^{-1/16}T$, this is at least $\frac45 t \nu_T$ when $T$ is large. Therefore $N(r(T)z_0,tT)\geq e^{\frac45 a(T)T\nu_T t}$ with high probability, and the second part of the lemma follows.

The third part is a consequence of the first part, since if $w$ maximizes $m_T(z,t)$, then $m_T(w,t)=\xi_T(w)(t-h_T(w))$ (so if $\xi_T(w)\leq \nu_T/2$ then $m_T(w,t)\leq \nu_T t/2 < \smfr{9}{10}t\nu_T$).
\end{proof}

Let
\[\kappa_T = \{y\in B(0,\rho_T r(T)) : \xi(y) \geq \nu_T a(T)/2\}.\]

\begin{lem}\label{nottoomanygoodpoints}
\[\P(\exists z\not\in B(0,\rho_T r(T)) : h_T(z)<t \hbox{ or } H_T(z)<t) \to 0\]
and
\[\P(\#\kappa_T > K_T ) \to 0\]
and 
\[ \p ( \exists z \in L_T(0,\rho_T) \, : \, \xi(z) \geq \nu_T^{-1} a(T) ) \ra 0 , \]
as $T\to\infty$.
\end{lem}

\begin{proof}
The first assertion holds by \cite[Proposition 4.9 and (7)]{OR14}. The second and third hold by \cite[Lemma 2.7(ii)]{OR14}.
\end{proof}

\begin{lem}\label{le:2204}
Let $w = w_T(t)$ be any point in $\Z^d$ such that $m_T(w/r(T),t)\geq m_T(z,t)$ for all $z\in L_T$. Then, for any $\eps > 0$
\[\P\bigg(\frac{\sum_{z\in B(w,\eps_T)} N(z,tT)}{\sum_{z\in \Z^d} N(z,tT)} \geq 1 - \eps\bigg) \to 1\]
and
\[\P(\exists z\in B(w,\eps_T )\setminus \{ w\} : \xi(z)\geq \nu_T a(T)/2) \to 0\]
as $T\to\infty$.
\end{lem}

\begin{proof}
The first statement is simply a rewording of \cite[Theorem 1.3]{OR14}. By Lemma \ref{whaslargexi} we may assume that $\xi(w)> \nu_T a(T)/2$, and by the first part of Lemma \ref{nottoomanygoodpoints} we may assume that $w\in B(0,r(T)\rho_T)$. The second statement then follows from the fact that
\begin{align*}
&\P(\exists y\in B(0,r(T)\rho_T) : \xi(y) > \nu_T a(T)/2, \hsl \exists z\in B(y,\eps_T) \hbox{ with } \xi(z) > \nu_T a(T)/2)\\
&\leq C_d r(T)^d \rho_T^d\P(\xi(0) > \nu_T a(T)/2, \hsl \exists z\in B(0,\eps_T) \hbox{ with } \xi(z) > \nu_T a(T)/2)\\
&\leq C_d r(T)^{d} \rho_T^d \cdot C_d \eps_T^d\cdot \P(\xi(0)>\nu_T a(T)/2)\cdot \P(\xi(0)>\nu_T a(T)/2)\\
&= C_d^2 (\smfr{3}{q})^d 2^{2\alpha} (\log\log T)^d \log^{-d/8}T
\end{align*}
which tends to $0$ as $T\to\infty$.
\end{proof}

Recall that
\[\kappa_T = \{y\in B(0,\rho_T r(T)) : \xi(y) \geq \nu_T a(T)/2\}\]
and $\theta_T = \nu_T^{2+2\alpha}a(T)$.
Moreover, we define
\[ \tilde L_\theta(tT) = \Big\{ y \in \kappa_T \,  : \, \#L_\theta(y,tT) \leq \xi(y)^{1/2}, tT - H(y) \in [0, \frac{1 + \tfrac 14 \log \xi(y)}{\xi(y)}] \Big\}, 
\]
as the points in $\kappa_T$ that get hit fairly late for the first time and the first particle is 
not followed by many other particles immediately afterwards.
Contributions from these points will be negligible and we first show how to control the points
in the complement.

\begin{lem}
For any $t\geq 0$,
\[\P(\exists y\in \kappa_T \cap \tilde L_{\theta_T}(tT)^c : N(y,tT; U_{y,\theta_T}) < \xi(y)^{1/10} \sum_{z\neq y} N(z,tT;U_{y,\theta_T},y)) \to 0\]
as $T\to\infty$.
\end{lem}

\begin{proof}
For any $y\in\kappa_T \cap \tilde L_{\theta_T}(tT)^c$, we have $\xi(y)\geq \nu_T a(T)/2$. Also, 
by Lemma~\ref{nottoomanygoodpoints}, we can assume that $\xi(y) \leq \nu_T^{-1} a(T)$ and so in particular $\theta_T > 10 d \xi(y)^{19/20}$
for $T$ sufficiently large.
Therefore we may apply Propositions \ref{notyprop} and \ref{bigenough} to see that
\begin{multline*}
\Px\bigg(\sum_{v\in L_{\theta_T}(y,tT)} N^v(y,tT;U_{y,\theta_T}) < \xi(y)^{1/10} \sum_{v\in L_{\theta_T}(y,tT)}\sum_{z\neq y} N^v(z,tT;U_{y,\theta_T})
\Big | \mathcal{G}_{L_{\theta_T}(y,t)} \bigg)\\
\leq \frac{2d}{(1-2^{-19/20})^2}\xi(y)^{-1/20} + (3+d)\xi(y)^{-1/16} + \xi(y)tT e^{-\xi(y)^{1/2}/16}.
\end{multline*}
But
\[\sum_{v\in L_{\theta_T}(y,tT)} N^v(y,tT;U_{y,\theta_T}) = N(y,tT;U_{y,\theta_T})\]
and
\[\sum_{v\in L_{\theta_T}(y,tT)}\sum_{z\neq y} N^v(z,tT;U_{y,\theta_T}) = \sum_{z\neq y} N(z,tT;U_{y,\theta_T},y),\]
so
\begin{multline*}
\Px\bigg(N(y,tT;U_{y,\theta_T}) < \xi(y)^{1/10} \sum_{z\neq y} N(z,tT;U_{y,\theta_T},y) \Big | \mathcal{G}_{L_{\theta_T}(y,t)} \bigg)\\
\leq \frac{2d}{(1-2^{-19/20})^2}\xi(y)^{-1/20} + (3+d)\xi(y)^{-1/16} + \xi(y)tT e^{-\xi(y)^{1/2}/16}.
\end{multline*}
By the second part of Lemma \ref{nottoomanygoodpoints}, we may assume that there are at most $K_T$ points in $\kappa_T$, and a union bound gives the result.
\end{proof}

Finally, we can control the points in $\tilde L_{\theta_T}(tT)$ that only get hit by a few particles that do not have much 
time to grow.

\begin{lem}
\[ \P \bigg( \exists y \in \tilde L_{\theta_T}(tT)  \, : \, 
N(y,tT, U_{y,\theta_T}) \geq \nu(T)^{-1} a(T) \bigg) \ra 0 . \]
\end{lem}

\begin{proof} Let   $y \in \tilde L_{\theta_T}(tT)$. We recall that then
$\xi(y) \geq \nu_T a(T) /2$, $t - H(y) \in [0,\frac{1 + \tfrac 14 \log \xi(y)}{\xi(y)}]$ and 
$\# L_{\theta_T}(y,tT) \leq \xi(y)^{1/2}$. 
Note that by Markov's inequality
\[ \begin{aligned} \Px \bigg( \sum_{v \in L_{\theta_T}(y,tT)} &  N^v (y,tT, U_{y,\theta_T}) \geq \xi(y) \, \Big| \, \mathcal{G}_{L_{\theta_T}(y,tT)} \bigg)\\
&  \leq \frac{1}{\xi(y)} \sum_{v \in L_{\theta_T}(y,tT)} \E [ N(y,s, U_{y,\theta_T} )]|_{s = t T - \tau_y(v) } \\
&  \leq \frac{1}{\xi(y)} \# L_{\theta_T}(y,tT) f(y, tT - H(y) ) \\
& \leq \frac{1}{\xi(y)} \xi(y)^{1/2}e^{\xi(y) \frac{1}{\xi(y)} (1 + \tfrac 14 \log \xi(y)) }
\leq \frac{e}{\xi(y)^{1/4}} 
\leq \frac{6}{\nu_T^{1/4} a(T)^{1/4}}. 
\end{aligned}
\]
By Lemma~\ref{nottoomanygoodpoints} we can assume that $\xi(y) \leq \nu_T^{-1} a(T)$
and that there  at most $K_T$ points in $\kappa_T$, 
so that a union bound gives the result.
\end{proof}

\begin{proof}[Proof of Theorem \ref{oneptthm}]
Let $\eps > 0$ and let $w = w_T(t)$ be any point in $\Z^d$ such that $m_T(w/r(T),t)\geq m_T(z,t)$ for all $z\in L_T$. Then by the previous lemmas, with high probability we know that the following events occur:
\begin{enumerate}[(i)]
\item There do not exist $z_1,z_2\in L_T(0,\rho_T)$ such that $z_1\neq z_2$, $\xi_T(z_1)\geq \nu_T/2$, $\xi_T(z_2)\geq \nu_T/2$ and $|\xi_T(z_1)-\xi_T(z_2)|\leq \nu_T^{2+2\alpha}$;
\item $\sum_{z\in\Z^d} N(z,tT; Q^c) \leq \frac 12 e^{\frac34 \nu_T a(T)tT}$, where $Q = \{z\in \Z^d : \xi(z)\leq \nu_T a(T)/2\}$;
\item For all $z\not\in B(0,\rho_T r(T))$, both $h_T(z)\geq t$ and $H_T(z)\geq t$;
\item $\sum_{z\in\Z^d} N(z,tT) \geq  e^{\frac45 a(T)Tt\nu_T}$;
\item $(\sum_{z\in B(w,\eps_T )} N(z,tT))/(\sum_{z\in \Z^d} N(z,tT)) >  1  - \eps/2$;
\item For all $z\in B(w,\eps_T ) \setminus \{ w\}$, we have $\xi(z)< \nu_T a(T)/2$;
\item For all $y\in\kappa_T \cap \tilde L_{\theta_T}(tT)^c$, we have $N(y,tT; U_{y,\theta_T}) \geq \xi(y)^{\frac{1}{10}} \sum_{z\neq y} N(z,tT;U_{y,\theta_T},y)$.
\item For all $y\in \tilde L_{\theta_T}(tT)$, we have $N(y,tT; U_{y,\theta_T}) \leq \nu_T^{-1} a(T)$.
\end{enumerate}

For $y\in\Z^d$, let
\[\tilde U_y = \{z\in B(0,r(T)\rho_T) : \xi(z)\leq \xi(y)\}.\]
Note that by (iii), any particle that is present at time $tT$ must either have remained within $Q$ or must have travelled via $y$ without exiting $\tilde U_y$ for some $y\in B(0,\rho_T r(T))$ with $\xi(y)\geq \nu_T a(T)/2$. By (i) and (iii), such a particle must in fact not have hit $U_{y,\theta}$. Thus for any $z\in\Z^d$,
\[N(z,tT) \leq N(z,tT;Q^c) + \sum_{y\in \kappa_T} N(z,tT ; U_{y,\theta},y).\]
By (ii), (vi), (vii) and (viii), we get
\[\sum_{z\in B(w,\eps_T )\setminus\{w\}} \hspace{-6mm} N(z,tT) \leq \frac 12 e^{\frac34 \nu_T a(T)tT} + C_d \eps_T^d\nu_T^{-1} a(T) + \Big(\frac{\nu_T a(T)}{2}\Big)^{-1/10}\sum_{y\in \kappa_T} N(y,tT; U_{y,\theta_T}).\]
But clearly
\[\sum_{y\in \kappa_T} N(y,tT; U_{y,\theta_T}) \leq \sum_{z\in\Z^d} N(z,tT),\]
and by (iv), for $T$ sufficiently large,
\[e^{\frac34 \nu_T a(T)tT} + C_d \eps_T^d \nu_T^{-1} a(T) \leq e^{-\nu_T a(T)tT/20} \sum_{z\in\Z^d} N(z,tT),\]
so
\[\sum_{z\in B(w,\eps_T)\setminus\{w\}} N(z,tT) \leq \left(e^{-\nu_T a(T)tT/20} + \Big(\frac{\nu_T a(T)}{2}\Big)^{-1/10}\right) \sum_{z\in \Z^d} N(z,tT).\]
Thus, for $T$ sufficiently large
\[\frac{\sum_{z\in B(w,\eps_T)\setminus\{w\}} N(z,tT)}{\sum_{z\in\Z^d} N(z,tT)} < \eps /2,\]
and combining this with (v),
\[\frac{\sum_{z\in \Z^d\setminus\{w\}} N(z,tT)}{\sum_{z\in\Z^d} N(z,tT)} < \eps . \]
Therefore we have shown that for any large $T$, with high probability, the site $w_T(t)$ satisfies
\[N(w_T(t),tT)/N(tT) > 1-\eps.\]
In particular taking $t=1$ completes the proof of Theorem \ref{oneptthm}.
\end{proof}


\bibliographystyle{alpha}

\end{document}